\documentclass[12pt, twoside]{article}
\usepackage{amsmath,amsthm,amssymb}
\usepackage{times}
\usepackage{enumerate}
\usepackage{url}

\theoremstyle{definition}
\newtheorem{thm}{Theorem}[section]
\newtheorem{cor}[thm]{Corollary}
\newtheorem{lem}[thm]{Lemma}

\newtheorem{rem}[thm]{Remark}

\numberwithin{equation}{section}

\newcommand{\subjclass}[1]{\bigskip\noindent\emph{2010 Mathematics Subject Classification:}\enspace#1}
\newcommand{\keywords}[1]{\noindent\emph{Keywords:}\enspace#1}

\newtheorem{proposition}{Proposition}[section]

\DeclareMathOperator{\dist}{dist}       

\newcommand{\R}{\mathbb{R}}             

\begin{document}


\baselineskip=17pt


\title{H\"older estimates for the Neumann problem in a domain with holes and a relation formula between the Dirichlet and Neumann problems}

\author{Victor Ca\~nulef-Aguilar\\
Facultad de Matem\'aticas, Pontificia Universidad Cat\'olica de Chile\\ Vicu\~na Mackenna 4860, Macul, Santiago, Chile\\
vacanulef@uc.cl \bigskip \\ 
Duvan Henao\\
Facultad de Matem\'aticas, Pontificia Universidad Cat\'olica de Chile\\ Vicu\~na Mackenna 4860, Macul, Santiago, Chile\\
dhenao@mat.puc.cl}

\date{19 June 2019}

\maketitle


\begin{abstract}

  In this paper we study the dependence of the H\"older estimates on the geometry of a domain with holes for the Neumann problem. For this, we study the H\"older regularity of the solutions of the Dirichlet and Neumann problems in the disk (and in the exterior of the disk), from which we get a relation between harmonic extensions and harmonic functions with prescribed Neumann condition on the boundary of the disk (for both interior and exterior problems).


\end{abstract}

\section{Introduction}

\subsection{Regularity constants}

We are interested in obtaining estimates for the Neumann problem, namely

\[ 
\begin{cases}
\Delta u =0 \text{ in } E,\\[0.5em]
\displaystyle \frac{\partial u}{\partial \nu}=g
 \text{ on } 
\partial E
\end{cases}
\]
and $\displaystyle \int_{E}u(y)dy=0$, for domains $E$ of the form:\\

\begin{align}\label{eq:genericE}
 E=B(z_0,r_0)\setminus\bigcup_{k=1}^{n}\overline{B(z_k,r_k)}\subset \R^2. 
\end{align}
Here, $\nu(x)$ is the unit outward normal, and  
$g\in C^{1,\alpha}\left( \bigcup_{k=0}^{n}\partial B(z_k, r_k)\right)$
for some  $\alpha\in(0,1)$.
The datum must be compatible with the equation:
\begin{equation}
\int_{\partial B(z_0,r_0)}g=\sum_{k=1}^{n}\int_{\partial B(z_k, r_k)}g. \label{nec}
\end{equation}

We find that the estimates do not blow up 
provided that the radii of the holes,
their distance to the outer boundary
and the distance between them 
do not become too small compared to the domain size.
To obtain quantitative estimates,
we assume throughout that
\begin{align} \label{eq:d}
  \begin{gathered}
  \forall i\geq 1\medspace r_i\geq d,\\
 \forall i\geq 1 \medspace  B(z_i,r_i+d)\subset B(z_0,r_0),
 \ \text{and}
 \\
 \min_{\substack{i,j\geq 1\\ i\neq j}}
  \dist(\overline {B(z_i,r_i)}, \overline{B(z_j,r_j)})\geq 2d,
  \end{gathered}
\end{align}
for some generic length $d$.
We also set
\begin{align}
  \label{def_C_P}
C_P(E):= \sup \left \{ \|\phi\|_{L^2(E)}: \phi \in H^1(E) \text{ s.t. } \|D\phi\|_{L^2(E)}=1
\ \text{and}\ \int_E \phi =0\right \},  
\end{align}

\begin{gather}\label{eq:B}
B  =B(E) := |E|^{\frac{1}{2}}C_P(E)\Big (d^{-\frac{1}{2}}C_P(E)+d^{\frac{1}{2}}\Big )n^{\frac{1}{2}}r_0^{\frac{1}{2}}.
\end{gather}

\begin{thm} \label{finalthm}
Let $B$ and $u$ be as above, then, we have:\\
$\left\|Du\right\|_{\infty(E)} \leq C (1+Bd^{-4}r_0)\left\|g\right\|_{\infty}+Cr_{0}^{\alpha}[g]_{0,\alpha}.$\\
$  [D u]_{0,\alpha(E)}\leq C(d^{-\alpha}+Bd^{-5}r_{0}^{2-\alpha})\left\|g\right\|_{\infty}+C[g]_{0,\alpha} .$\\
$ \left\|D^2 u\right\|_{\infty(E)} \leq C(d^{-1}+Bd^{-5}r_0)\left\|g\right\|_{\infty}+Cd^{\alpha-1}[g]_{0,\alpha}+C\left\|g'\right\|_{\infty}+Cr_{0}^{\alpha}[g']_{0,\alpha} . $\\
$  [D^2 u]_{0,\alpha(E)}
\leq C(d^{-1-\alpha}+Bd^{-6}r_{0}^{2-\alpha})\left\|g\right\|_{\infty}+Cd^{-1}[g]_{0,\alpha}+Cd^{-\alpha}\left\|g'\right\|_{\infty}+C[g']_{0,\alpha}.$
\end {thm}

It should be noted that the above theorem shows the dependence on $d$ and $r_0$ of the elliptic regularity constant in front of each of the seminorms $\left\|g\right\|_{\infty}$, $[g]_{0,\alpha}$, $\left\|g'\right\|_{\infty}$ and $[g']_{0,\alpha}$ separately, as opposed to just an estimate of the form $ \left\|D^2 u\right\|_{2,\alpha}\leq C(d,r_0)\left\|g\right\|_{2,\alpha}$, where much information is lost. This more complete understanding of the regularity theory is interesting in itself and might be relevant in applications. In particular, independent knowledge of the dependence on the derivatives of different order is necessary in any careful analysis of the scalings of the problem.

As can be seen, here we give a deeper treatment to the regularity constants than the one that is done in \cite{CH19a}. For instance, in the proof of Lemma \ref{lemma1}, the estimates of the form $[\cdot]_{0,\alpha}\leq C(d,r_{max})\| \cdot \|_{L^1}$, show the dependence of $C(d,r_{max})$ on $d$ and $R$ in a more explicit and detailed way (compare with the proof of \cite[Lemma 5.4]{CH19a}).

The motivation for studying the above problem follows by a cavitation problem analyzed in \cite{CH19a}, where the major difficulty relies in constructing a family of explicit admissible deformation maps producing round cavities of a certain size. For that, near each cavity point, one can define explicitly a radially symmetric deformation map creating cavities of the desired size. Now, for gluing the above we can use the flow of Dacorogna and Moser \cite{DaMo90}, which yields the following free boundary equation

\begin{align} 
\left\{ \begin{aligned}
div \medspace v_t&=0& &\text{in $ E(t)$,} \\
v_t(x) &=g_t(x)\nu(x) & &\text{on $ \partial E(t)$ ,}
\end{aligned} \right. 
\end {align}
where $E(t)=B(z_0(t),r_0(t))\setminus\bigcup_{k=1}^{n}\overline{B(z_k(t),r_k(t))}\subset \R^2.$ Reducing the problem (after using a Leray type descomposition) to the study of the regularity of the solution to the Neumann problem in $E(t)$. Evidently, if we want to estimate $\left\|v_t\right\|_{\infty}$ and $\left\|Dv_t\right\|_{\infty}$ through the evolution, we have to take care of the uniform control in time of each seminorm, so we need to know the dependence on the domain of each regularity constant.

One could think that the circular shape is too restrictive for modelling cavitation phenomena, but it seems that the minimizers prefer to keep their round shaped cavities (until a critical load) as suggested in \cite{BaMu84} and \cite{HeSe13}. That makes the problem with circular holes, which is already challenging, also interesting at least in that application. From the more pure side, working with holes that are circular allows for fine and more explicit calculations using singular integrals, leading to a better understanding of the dependence on the geometry.

\subsection{Relation between the Dirichlet and Neumann problems}

The other result is the relation between harmonic extensions and harmonic functions with prescribed Neumann data. First, let us introduce the two fundamental kernels :

\begin{equation}\label{kernelp}
  P_{r}(\phi)=\frac{1-r^2}{r^2+1-2r\cos(\phi)}  
  \end{equation}
  \begin{equation}\label{kernelk}
K_{r}(\phi)=\frac{r\sin\left(\phi\right)}{r^2+1-2r\cos(\phi)} . 
\end{equation}

Now, let us recall that on the disk, the solution of the Dirichlet problem, namely: 

\begin{align} \label{bvpd}
\left\{ \begin{aligned}
\Delta u&=0& &\text{in $ B(0,1)$,} \\
u &=g & &\text{on $ \partial B(0,1)$,}
\end{aligned} \right. 
\end {align}

is given by: 
$$u(re^{i\phi})=\frac{1}{2\pi}P_r*g(\phi)=\frac{1}{2\pi}\int_{-\pi}^{\pi}\frac{1-r^2}{1+r^2-2r\cos(\tau-\phi)}g(\tau)d\tau,$$

and the solution of the Neumann problem (with zero average): 

\begin{align}  \label{bvpn1}
\left\{ \begin{aligned}
\Delta w&=0& &\text{in $ B(0,1)$,} \\
\frac{\partial w}{\partial \nu} &=g& &\text{on $ \partial B(0,1)$},\\
\int_{B(0,1)}w dx& =0, 
\end{aligned} \right. 
\end {align}

is equal to: 

$$w(x)=-\frac{1}{\pi}\int_{-\pi}^{\pi}log|x-y|g(\tau)d\tau \text{, where $y=(cos(\tau),sin(\tau))$}.$$

In particular, the solution of the problem:

\begin{align} \label{bvpn2}
\left\{ \begin{aligned}
\Delta \omega&=0& &\text{in $ B(0,1)$,} \\
\frac{\partial \omega}{\partial \nu} &=g'& &\text{on $ \partial B(0,1)$},\\
\int_{B(0,1)}\omega dx& =0. 
\end{aligned} \right. 
\end {align}

Where $g'$ denotes the tangential derivative of $g$, is given by: 

$$\omega(x)=-\frac{1}{\pi}\int_{-\pi}^{\pi}log|x-y|g'(\tau)d\tau=\frac{1}{\pi}K_r*g(\phi)=\frac{1}{\pi}\int_{-\pi}^{\pi}\frac{r\sin\left(\tau-\phi\right)}{r^2+1-2r\cos(\tau-\phi)}g(\tau)d\tau.$$ 

\begin{rem}

The solutions to the exterior problem are very similar.

\end{rem}

The analysis in both \cite{CH19a} and Theorem \ref{finalthm} is made possible by the following more fundamental connection between the Dirichlet and Neumann problems, a result that is of independent interest and we highlight as our second main theorem.

\begin{thm} \label{rfthm}

Let $u$ and $w$ be the unique solutions to \eqref{bvpd} and \eqref{bvpn1} respectively, then:

$$D u(re^{i\phi})=-\frac{1}{r}\left(\frac{1}{\pi}K_r*g'(\phi)\right)e^{i\phi}+\frac{1}{r}\left(\frac{1}{2\pi}P_r*g'(\phi)\right)e^{i\left(\phi+\frac{\pi}{2}\right)},$$

$$D w(re^{i\phi})=\frac{1}{r}\left(\frac{1}{2\pi}P_r*g(\phi)\right)e^{i\phi}+\frac{1}{r}\left(\frac{1}{\pi}K_r*g(\phi)\right)e^{i\left(\phi+\frac{\pi}{2}\right)}.$$

\end {thm}

\begin{cor}

If $u$ and $\omega$ are the solutions to \eqref{bvpd} and \eqref{bvpn2}, then

$$Du(re^{i\phi})=D\omega(re^{i\phi})\cdot e^{i\frac{\pi}{2}} \text{ \medspace$\forall r\in (0,1)$ \medspace $\forall \phi \in\mathbb{R}$}.$$

\end{cor}

If $u=g$ on $\partial B(0,1)$ then the tangential derivative of $u$ is $g'$, that is, it coincides with the normal derivative of $\omega$. This somehow suggests that on $\partial B(0,1)$ the result that $Du$ and $D\omega$ are the same up to a rotation by $\frac{\pi}{2}$ is to be expected. However, it is surprising that the connection carries through to the interior of the domain.

\begin{rem}

The analogous formulas for the exterior problem also hold.

\end{rem}

\begin{rem}
One could think about the previous corollary as an analogous for the Cauchy-Riemann equations. More precisely, if $f(x+iy)=u(x,y)+iv(x,y)$, is analytic in the disk, then  the Cauchy-Riemann equations are equivalent to $Du=e^{\frac{3}{2}i\pi}Dv$. So, if $u$ and $\omega$ are the solutions to \eqref{bvpd} and \eqref{bvpn2}, then $f=u-i\omega$ is analytic in the disk. Moreover:
$$f(z)=\frac{1}{2\pi}\int_{\mathbb{T}}\frac{e^{i\tau}+z}{e^{i\tau}-z}g(\tau)d\tau,$$ which is the Schwarz integral formula for $g$ (i.e. an holomorphic function whose real part on the boundary is equal to $g$).
The point is, that we can see from the relation formula that the imaginary part is related to the solution of the Neumann problem. Actually, the last holds for every smooth domain: To see this, it suffices to note that if $f$ is holomorphic in a smooth domain, then from the Cauchy-Riemann equations we get: $$ -\partial_{\nu}\left( Im(f)\right)= \langle\,-D(Im(f)),\nu \rangle =\langle\,-i\cdot D(Re(f)),-i\cdot\tau \rangle=\langle\,D(Re(f)),\tau \rangle = \partial_{\tau}\left( Re(f)\right),$$
where $\partial_{\nu}$ and $\partial_{\tau}$ are the normal and tangential derivatives. 
\end{rem}

From the relation formulas, we deduce that we can study the regularity of the above convolutions to obtain the regularity of the harmonic functions.

\section{Notation and Preliminaries}

\subsubsection*{Function spaces and Green's function}

We fix a value of $\alpha\in (0,1)$ and work with the norms 
$\left\|f\right\|_{\infty}:=\sup|f(x)|$ and
\begin{align*}
 [f]_{0,\alpha}&:=\sup_{x\neq y}\frac{|f(x)-f(y)|}{|x-y|^{\alpha}}, 
 & 
 \left\|f\right\|_{0,\alpha} &:=\left\|f\right\|_{\infty}+[f]_{0,\alpha},\\
 [f]_{1,\alpha}&:=\sup_{x\neq y}\frac{|Df(x)-Df(y)|}{|x-y|^{\alpha}},
 &
 \left\|f\right\|_{1,\alpha}&:=\left\|f\right\|_{\infty}+\left\|Df\right\|_{\infty}+[f]_{1,\alpha}.
\end{align*}
The function $g$ will belong to
$$C_{per}^{0,\alpha}:=\{ g\in C_{loc}^{0,\alpha}(\mathbb{R}): g \text{ is $2\pi$-periodic} \}.$$
The inversion of $x\in \R^2$ with respect to $B(0,R)$ is 
$x^{*}=\frac{R^2}{|x|^2}x.$ Set
$$\Phi(x) :=\frac{-1}{2\pi}\log(|x|), \quad
\phi^x(y):=\frac{1}{2\pi}\log(|y-x^{*}|)-\frac{|y|^2}{4\pi R^2},
\quad
G_{N}(x,y):=\Phi (x)-\phi^x(y).$$
The expression $u_{,\beta}$ stands for $\partial_{\beta} u=\frac{\partial u}{\partial x_\beta}$.\\

\section{Relation Formula}

\begin{proof}[Proof of the Theorem \ref{rfthm}]

Set $x=re^{i\phi}\in B(0,1)$ and $y=e^{i\tau}=(\cos(\tau),\sin(\tau))$. Let us prove the first formula. For that, let us start by computing the $x$ derivative of the Poisson kernel: 
$$D_x(P_r(\tau-\phi))=D_x\left( \frac{1-|x|^2}{|x-y|^2} \right)=-2\left( \frac{x(|x-y|^2+1-|x|^2)-y(1-|x|^2)}{|x-y|^4}\right) .$$
Now, for $x\in B(0,1)$, we have (due to the dominated convergence theorem):
$$D_x(u)=\frac{1}{2\pi}\int_{-\pi}^{\pi}D_x\left(P_r(\tau-\phi)\right)g(\tau)d\tau  .$$
In addition, the $x$ derivatives of $P_r(\tau-\phi)$ are given by (note that we use $\tau=(\tau-\phi)+\phi$ and $|x-y|^2=1+r^2-2r\cos(\tau-\phi)$):
$$\frac{\partial }{\partial x_1}\left(P_r(\tau-\phi)\right)=-2\frac{\cos(\phi)(2r-(r^2+1)\cos(\tau-\phi))+\sin(\phi)(1-r^2)\sin(\tau-\phi)}{(1+r^2-2r\cos(\tau-\phi))^2}$$
$$ \frac{\partial }{\partial x_2}\left(P_r(\tau-\phi)\right)=-2\frac{\sin(\phi)(2r-(r^2+1)\cos(\tau-\phi))-\cos(\phi)(1-r^2)\sin(\tau-\phi)}{(1+r^2-2r\cos(\tau-\phi))^2}.$$
Furthermore:
$$\int_{-\pi}^{\pi}\frac{2r-(r^2+1)\cos(\tau-\phi)}{(1+r^2-2r\cos(\tau-\phi))^2}g(\tau)d\tau=-\int_{-\pi}^{\pi}\frac{d}{d\tau}\left(\frac{\sin(\tau-\phi)}{1+r^2-2r\cos(\tau-\phi)}\right)g(\tau)d\tau$$
$$=\int_{-\pi}^{\pi}\frac{\sin(\tau-\phi)}{1+r^2-2r\cos(\tau-\phi)}g'(\tau)d\tau=\int_{-\pi}^{\pi}\frac{\sin(\tau)}{1+r^2-2r\cos(\tau)}g'(\tau+\phi)d\tau.$$
Moreover:
$$ \int_{-\pi}^{\pi}\frac{(1-r^2)\sin(\tau-\phi)}{(1+r^2-2r\cos(\tau-\phi))^2}g(\tau)d\tau$$
$$=-\frac{1-r^2}{2r}\int_{-\pi}^{\pi}\frac{d}{d\tau}\left(\frac{1}{1+r^2-2r\cos(\tau-\phi)}\right)g(\tau)d\tau$$
$$=\frac{1}{2r}\int_{-\pi}^{\pi}\frac{1-r^2}{1+r^2-2r\cos(\tau-\phi)}g'(\tau)d\tau.$$
From the above, it is easy to conclude the validity of the first formula.\\
For proving the second formula, first note that the derivative of $w$ (times $-\pi$) is given by:
$$-\pi Dw(x)=\int_{-\pi}^{\pi}g(\tau)\frac{x-y}{|x-y|^2}d\tau.$$
Now, the tangential component is equal to :

$$\int_{-\pi}^{\pi}g(\tau)\frac{-\cos\left(\tau-\phi-\frac{\pi}{2}\right)}{r^2+1-2r\cos(\tau-\phi)})d\tau=-\frac{1}{r}\int_{-\pi}^{\pi}g(\tau)\frac{r\sin\left(\tau-\phi\right)}{r^2+1-2r\cos(\tau-\phi)}d\tau$$

On the other hand, the normal component is equal to :

$$ \int_{-\pi}^{\pi}g(\tau)\frac{r-\cos(\tau-\phi)}{r^2+1-2r\cos(\tau-\phi)}d\tau=\frac{1}{r}\int_{-\pi}^{\pi}g(\tau)\frac{r^2-1}{r^2+1-2r\cos(\tau-\phi)}d\tau$$
$$-\int_{-\pi}^{\pi}g(\tau)\frac{r-\cos(\tau-\phi)}{r^2+1-2r\cos(\tau-\phi)}d\tau+\frac{1}{r}\int_{-\pi}^{\pi}g(\tau)d\tau.$$

From which the result follows by using that the integral of $g$ is equal to zero because $w$ is harmonic.

\end{proof}

\section{H\"older regularity of the convolutions}

\begin{lem} \label{lemma3}
Let $g\in C_{per}^{0,\alpha}$, $\phi\in[0,2\pi]$, $1<r_2<r_1$. 
Then: $$ |\omega(r_1e^{i\phi})-\omega(r_2e^{i\phi})|\leq Cr_1[g]_{0,\alpha}|r_1-r_2|^{\alpha},$$
 where 
\begin{equation} \label{kernel1}
\omega:=\int_{-\pi}^{\pi}g(\tau+\phi)\frac{r\sin(\tau)d\tau}{r^2+1-2r\cos(\tau)}
\end{equation}
\end {lem}

\begin{proof} 
Note that:
$$|\omega(r_1e^{i\phi})-\omega(r_2e^{i\phi})|=\left|\int_{r_2}^{r_1}\frac{\partial \omega}{\partial r}dr\right|\leq \int_{r_2}^{r_1}\left|\frac{\partial \omega}{\partial r}\right|dr.$$
On the other hand:$$ \frac{\partial \omega}{\partial r}(re^{i\phi})=\int_{-\pi}^{\pi}g(\tau+\phi)\frac{(1-r^2)\sin(\tau)d\tau}{((1-r)^2+2r(1-\cos(\tau)))^2}$$
$$=\int_{-\pi}^{\pi}(g(\tau+\phi)-g(\phi))\frac{(1-r^2)\sin(\tau)d\tau}{((1-r)^2+2r(1-\cos(\tau)))^2},$$
where we have used that $\sin(\tau)$ is odd. Moreover:
$$\left|\int_{|\tau|\leq r-1}(g(\tau+\phi)-g(\phi))\frac{(1-r^2)\sin(\tau)d\tau}{((r-1)^2+2r(1-\cos(\tau)))^2}\right|$$
$$\leq\int_{|\tau|\leq r-1}\frac{2r_1(r-1)[g]_{0,\alpha}|\tau|^{1+\alpha}}{((r-1)^2+2r(1-\cos(\tau)))^2}\leq \int_{|\tau|\leq r-1}\frac{Cr_1[g]_{0,\alpha}(r-1)^{2+\alpha}}{(r-1)^4}d\tau$$
$$=Cr_1[g]_{0,\alpha}(r-1)^{\alpha-1}.$$

Recall that $\frac{2}{\pi^2}|\tau|^2\leq 1-\cos(\tau) \leq \frac{1}{2}|\tau|^2$ for $\tau\in (-\pi,\pi)$. To estimate the rest of the integral, it suffices to note that:
$$\left|\int_{ r-1 \leq|\tau|\leq \pi}(g(\tau+\phi)-g(\phi))\frac{(1-r^2)\sin(\tau)d\tau}{((r-1)^2+2r(1-\cos(\tau)))^2}\right|$$
$$\leq\int_{r-1\leq |\tau|\leq \pi}2r_1(r-1)[g]_{0,\alpha}\frac{|\tau|^{1+\alpha}}{((r-1)^2+2r(1-\cos(\tau)))^2}d\tau$$
$$\leq \int_{r-1\leq|\tau|\leq \pi}Cr_1(r-1)[g]_{0,\alpha}\frac{|\tau|^{1+\alpha}}{4|\tau|^4}d\tau \leq (r-1)Cr_1[g]_{0,\alpha}\int_{r-1\leq|\tau|\leq \pi}|\tau|^{\alpha-3}d\tau $$
$$\leq Cr_1(r-1)(r-1)^{\alpha-2}= Cr_1[g]_{0,\alpha}(r-1)^{\alpha-1}.$$

Finally:
$$|\omega(r_1e^{i\phi})-\omega(r_2e^{i\phi})|
\leq \int_{r_2}^{r_1}\left|\frac{\partial \omega}{\partial r}\right|dr\leq Cr_1[g]_{0,\alpha}\int_{r_2}^{r_1}(r-1)^{\alpha-1}dr\leq Cr_1[g]_{0,\alpha}|r_1-r_2|^{\alpha}.$$
(Recall that $|x|^{\alpha}$ is locally H\"older continuous in $[0,\infty)$.)
\end{proof}

\begin{lem} \label{lemma4}
Let $g\in C_{per}^{0,\alpha}$, $r>1$, $\omega$ as in \eqref{kernel1}, and $x_1,x_2\in\mathbb{R}^2$ such that $|x_1|=|x_2|=r$. Then:
$$|\omega(x_1)-\omega(x_2)|\leq Cr^2[g]_{0,\alpha}(r-1)^{\alpha-1}|x_1-x_2|.$$
\end {lem}

\begin{proof}
  Let $1<r\leq 2$ and $|\phi_1-\phi_2|\leq\pi$, if we define $K_r(\tau)=\frac{\sin(\tau)}{1+r^2-2r\cos(\tau)}$ then: 
$$\omega(re^{i\phi})=r\int_{-\pi}^{\pi}g(\tau+\phi)K_r(\tau)d\tau=-r\int_{-\pi}^{\pi}g(\tau)K_r(\phi-\tau)d\tau.$$ 

 The derivative of $K_r$ is given by:
$$ \frac{\cos(\tau)(1+r^2)-2r}{(1+r^2-2r\cos(\tau))^2}=\left(1-\frac{(1+r)^2(1-\cos(\tau))}{(r-1)^2+2r(1-\cos(\tau))}\right)(1+r^2-2r\cos(\tau))^{-1}.$$

Since:
$$ \left|\frac{\cos(\tau)(1+r^2)-2r}{(r-1)^2+2r(1-\cos(\tau))}\right|\leq 1+\frac{(1+r)^2(1-\cos(\tau))}{2r(1-\cos(\tau))}\leq Cr,$$
\noindent we have:

$$\left|\frac{\partial K_r}{\partial \tau}(\tau)\right|\leq \frac{Cr}{(r-1)^2+2r(1-\cos(\tau))}\leq C'r|\tau|^{-2}, \text{if }|\tau|\leq \pi.$$

Let $\rho=|\phi_1-\phi_2|\leq \pi$, then:
$$\left|\frac{\partial \omega}{\partial \phi}\right|\leq r\left|\int_{-\pi}^{\pi}(g(\tau)-g(\phi))K_r'(\phi-\tau)d\tau\right|$$
$$ \leq    Cr^2[g]_{0,\alpha}\int_{|\tau-\phi|\leq r-1}\frac{|\tau-\phi|^{\alpha}}{(r-1)^2}d\tau+ Cr^2[g]_{0,\alpha}\int_{r-1\leq|\tau-\phi|\leq \pi}|\phi-\tau|^{\alpha-2}d\tau$$
$$\leq Cr^2(r-1)^{\alpha-1}[g]_{0,\alpha}. $$
Now using the fundamental theorem of calculus:
$$|\omega(re^{i\phi_1})-\omega(re^{i\phi_2})|\leq \int_{\phi_1}^{\phi_2}Cr^2(r-1)^{\alpha-1}[g]_{0,\alpha}d\phi$$
$$=Cr^2(r-1)^{\alpha-1}[g]_{0,\alpha}|\phi_1-\phi_2|\leq Cr^2(r-1)^{\alpha-1}[g]_{0,\alpha}|re^{i\phi_1}-re^{i\phi_2}|.$$
\end{proof}

\begin{proposition} \label{prop5}
Let $g\in C_{per}^{0,\alpha}$, $\omega$ as in \eqref{kernel1}, and $x_1,x_2\in\mathbb{R}^2$ such that $ 1<|x_2|\leq|x_1|\leq 2$. Then:
$$|\omega(x_1)-\omega(x_2)|\leq C[g]_{0,\alpha}|x_1-x_2|^{\alpha}.$$
(i.e. $[\omega]_{0,\alpha}\leq C[g]_{0,\alpha}$).
\end {proposition}

\begin{proof}
 Set $x_1=r_1e^{i\phi_1}$, $x_2=r_2e^{i\phi_2}$, $|\phi_1-\phi_2|\leq \pi$, $\rho:=|x_1-x_2|$.
 
 \begin{description}
  \item[Case $r_1-1\geq \rho$:] by Lemmas \ref{lemma3} and \ref{lemma4} :
$$ |\omega(x_1)-\omega(x_2)|\leq |\omega(r_1e^{i\phi_1})-\omega(r_1e^{i\phi_2})|+|\omega(r_1e^{i\phi_2})-\omega(r_2e^{i\phi_2})|$$
$$\leq Cr_1[g]_{0,\alpha}(r_1-1)^{\alpha-1}|r_1e^{i\phi_1}-r_1e^{i\phi_2}|+Cr_1[g]_{0,\alpha}||x_1|-|x_2||^{\alpha}$$
$$  \leq  2C[g]_{0,\alpha}\rho^{\alpha-1}(|r_1e^{i\phi_1}-r_2e^{i\phi_2}|+|r_2e^{i\phi_2}-r_1e^{i\phi_2}|)+2C[g]_{0,\alpha}|x_1-x_2|^{\alpha} $$
$$  \leq   C[g]_{0,\alpha}(\rho^{\alpha-1}(\rho+\rho)+\rho^{\alpha}).$$

  \item [Case $r_1-1< \rho$:] set $r:=1+\rho$. Note that since $r_2<r_1<2$, then $r=1+|x_1-x_2|<1+r_1+r_2\leq 5$ 
$$ |\omega(x_1)-\omega(x_2)|\leq |\omega(r_1e^{i\phi_1})-\omega(re^{i\phi_1})|+|\omega(re^{i\phi_1})-\omega(re^{i\phi_2})|+|\omega(re^{i\phi_2})-\omega(r_2e^{i\phi_2})|  $$
$$ \leq  2\cdot 5C[g]_{0,\alpha}|r-r_1|^{\alpha}+5C[g]_{0,\alpha}(r-1)^{\alpha-1}|re^{i\phi_1}-re^{i\phi_2}|, $$
since $r_2>1$, then $r-r_2=\rho-(r_2-1)<\rho$. On the other hand: $|re^{i\phi_1}-re^{i\phi_2}|\leq |r-r_1|+|x_1-x_2|+|r_2-r|<3\rho$
 and $(r-1)^{\alpha-1}=\rho^{\alpha-1}$ by definition of $r$. This completes the proof.
  \end{description}

\end{proof}

\begin{proposition} \label{prop6}
Let $g\in C_{per}^{0,\alpha}$, $\omega$ as in \eqref{kernel1}, and $x_1,x_2\in\mathbb{R}^2$ such that $ 1<|x_2|\leq|x_1|\leq 2$. Then:
$$ \left\|\omega\right\|_{\infty}\leq C[g]_{0,\alpha} .$$
\end {proposition}

\begin{proof}
 It is easy to see that:
$$ |\omega|\leq C[g]_{0,\alpha}\int_{-\pi}^{\pi}\frac{|\tau|^{1+\alpha}}{|\tau|^2}d\tau\leq C[g]_{0,\alpha}.$$
\end{proof}

\begin{lem}  \label{lemma5}
Let $x=re^{i\phi}$ and $y=e^{i\tau}$. Let $u$ be given by:
 \begin{equation}\label{PKer}
u(re^{i\phi})=\frac{1-r^2}{2\pi}\int_{-\pi}^{\pi}\frac{g(\tau)d\tau}{|x-y|^2},
\end{equation}
then: $\left\|u\right\|_{\infty}\leq C\left\|g\right\|_{\infty}.$
\end {lem}

\begin{proof}
 This is immediate from the well-known formula (see \cite{Gamelin01}):
 \begin{equation}\label{Poisson}
\frac{r^2-1}{2\pi}\int_{-\pi}^{\pi}\frac{d\tau}{1+r^2-2r\cos(\tau)}=sgn(r-1).
\end{equation}
\end{proof}

\begin{lem}   \label{lemma6}
Let $g\in C_{per}^{0,\alpha}$, $r>1$, $|\phi_1-\phi_2|\leq \pi$ and $u$ as in \eqref{PKer}. Then:
$$|u(re^{i\phi_1})-u(re^{i\phi_2})|\leq C[g]_{0,\alpha}|re^{i\phi_1}-re^{i\phi_2}|. $$
\end {lem}

\begin{proof}
 First note that (thanks to \eqref{Poisson}):
$$u(re^{i\phi})=\frac{1-r^2}{2\pi }\int_{-\pi}^{\pi}g(\tau)\frac{d\tau}{|x-y|^2}=\frac{1-r^2}{2\pi }\int_{-\pi}^{\pi}\frac{g(\tau+\phi)-g(\phi) }{1+r^2-2r\cos(\tau)}d\tau-g(\phi),$$
\noindent then:
$$|u(re^{i\phi_1})-u(re^{i\phi_2})|\leq [g]_{0,\alpha}|\phi_1-\phi_2|^{\alpha}+\frac{r^2-1}{2\pi }\int_{-\pi}^{\pi}\frac{|g(\tau+\phi_1)-g(\tau+\phi_2)|}{1+r^2-2r\cos(\tau)}d\tau$$
$$\leq [g]_{0,\alpha}|\phi_1-\phi_2|^{\alpha}+[g]_{0,\alpha}|\phi_1-\phi_2|^{\alpha}\frac{r^2-1}{2\pi}\frac{2\pi}{r^2-1}\leq C'[g]_{0,\alpha}|re^{i\phi_1}-re^{i\phi_2}|^{\alpha}.$$
\end{proof}

\begin{lem}    \label{lemma7}
Let $g\in C_{per}^{0,\alpha}$, $u$ as in \eqref{PKer}, $1<r_2<r_1\leq 2$. Then:
$$ |u(r_1e^{i\phi})-u(r_2e^{i\phi})|\leq C[g]_{0,\alpha}|r_1-r_2|^{\alpha}.$$
\end{lem}

\begin{proof}
 Note that:
$$ \frac{d}{dr}\left( \frac{1-r}{1+r^2-2r\cos(\tau)} \right)=\frac{(r-1)^2-2(1-\cos(\tau))}{((r-1)^2+2r(1-\cos(\tau)))^2},$$
also:
$$\frac{d }{dr}\left(\frac{(1+r)(1-r)}{(1-r)^2+2r(1-\cos(\tau))}\right)=(1+r)\frac{d}{dr}\left( \frac{1-r}{1+r^2-2r\cos(\tau)} \right)$$
$$+ \frac{1-r}{1+r^2-2r\cos(\tau)}. $$
We want to prove $\left|\frac{\partial u}{\partial r}\right|\leq C(r-1)^{\alpha-1}$, for $r\in (1,2)$. For that, it suffices to estimate the following integrals:
$$ \left|(r-1)\int_{-\pi}^{\pi}(g(\tau+\phi)-g(\phi))\frac{d\tau}{(r-1)^2+2r(1-\cos(\tau))}\right|\leq C\pi^{\alpha}[g]_{0,\alpha}(r-1)\frac{2\pi}{r^2-1}$$
$$\leq C[g]_{0,\alpha} \leq C[g]_{0,\alpha}(r-1)^{\alpha-1}.$$
Now let us estimate the second integral for $|\tau|\leq r-1$:
$$ 2\left|\int_{|\tau|\leq r-1}(g(\tau+\phi)-g(\phi))\frac{1-\cos(\tau)}{((r-1)^2+2r(1-\cos(\tau)))^2}d\tau\right|$$
$$\leq C[g]_{0,\alpha}\int_{|\tau|\leq r-1}\frac{|\tau|^{\alpha+2}}{((r-1)^2+2r(1-\cos(\tau)))^2}d\tau$$
$$\leq C[g]_{0,\alpha}\int_{|\tau|\leq r-1}\frac{|\tau|^{\alpha+2}}{(r-1)^4}d\tau\leq C'[g]_{0,\alpha}\frac{(r-1)^{\alpha+3}}{(r-1)^4}=C'[g]_{0,\alpha}(r-1)^{\alpha-1}.$$
Then for $r-1\leq |\tau|\leq \pi$:
$$ 2\left|\int_{r-1\leq |\tau|\leq \pi}(g(\tau+\phi)-g(\phi))\frac{1-\cos(\tau)}{((r-1)^2+2r(1-\cos(\tau)))^2}d\tau\right|$$
$$\leq [g]_{0,\alpha}C\int_{r-1\leq |\tau|\leq \pi}\frac{|\tau|^{\alpha+2}}{(2|\tau|^2)^2}d\tau\leq C'((r-1)^{\alpha-1}-\pi^{\alpha-1})\leq C'[g]_{0,\alpha}(r-1)^{\alpha-1}.$$
Finally, let us estimate the last integral for $|\tau|\leq r-1$:
$$ (r-1)^2\left|\int_{|\tau|\leq r-1}\frac{g(\tau+\phi)-g(\phi)}{((r-1)^2+2r(1-\cos(\tau)))^2}d\tau\right|$$
$$\leq [g]_{0,\alpha}C(r-1)^2\int_{|\tau|\leq r-1}\frac{|\tau|^{\alpha}}{(r-1)^4}d\tau\leq C'[g]_{0,\alpha}(r-1)^{\alpha-1}.$$

At last for $r-1\leq |\tau|\leq \pi$:

$$  (r-1)^2\left|\int_{r-1\leq |\tau|\leq \pi}\frac{g(\tau+\phi)-g(\phi)}{((r-1)^2+2r(1-\cos(\tau)))^2}d\tau\right|$$
$$ \leq  C[g]_{0,\alpha}(r-1)^2\int_{r-1\leq |\tau|\leq \pi}\frac{|\tau|^{\alpha}}{|\tau|^4}d\tau\leq C'[g]_{0,\alpha}(r-1)^2((r-1)^{\alpha-3}-\pi^{\alpha-3})$$
$$\leq C'[g]_{0,\alpha}(r-1)^{\alpha-1}.$$

In conclusion, we have:
$$|u(r_1e^{i\phi})-u(r_2e^{i\phi})|=\left|\int_{r_2}^{r_1}\frac{\partial u}{\partial r}dr\right|\leq \int_{r_2}^{r_1}\left|\frac{\partial u}{\partial r}\right|dr\leq C[g]_{0,\alpha}\int_{r_2}^{r_1}(r-1)^{\alpha-1}dr$$
$$\leq C'[g]_{0,\alpha}|r_1-r_2|^{\alpha},$$
and the result follows from the above.
\end{proof}

\begin{proposition}   \label{prop7}
Let $g\in C_{per}^{0,\alpha}$, $u$ as in \eqref{PKer} $1<r_1\leq r_2\leq 2$, and $|\phi_1-\phi_2|\leq \pi$. Then:
$$ |u(r_1e^{i\phi_1})-u(r_2e^{i\phi_2})|\leq C[g]_{0,\alpha}|r_1e^{i\phi_1}-r_2e^{i\phi_2}|^{\alpha}.$$
(i.e. $[u]_{0,\alpha(B(0,2)\setminus B(0,1))}\leq C[g]_{0,\alpha (\partial B(0,1))}$).
\end {proposition}

\begin{proof}
 Note that from the previous propositions we get:
$$ |u(r_1e^{i\phi_1})-u(r_2e^{i\phi_2})|\leq |u(r_1e^{i\phi_1})-u(r_1e^{i\phi_2})|+|u(r_1e^{i\phi_2})-u(r_2e^{i\phi_2})|$$
$$\leq C[g]_{0,\alpha (\partial B(0,1))}|r_1e^{i\phi_1}-r_1e^{i\phi_2}|^{\alpha}+C[g]_{0,\alpha (\partial B(0,1))}|r_1e^{i\phi_2}-r_2e^{i\phi_2}|^{\alpha}  $$
$$\leq C[g]_{0,\alpha (\partial B(0,1))}| r_1e^{i\phi_1}- r_2e^{i\phi_2}|^{\alpha}+C[g]_{0,\alpha (\partial B(0,1))}\left|r_2-r_1 \right|^{\alpha}$$
$$\leq C[g]_{0,\alpha (\partial B(0,1))}| r_1e^{i\phi_1}- r_2e^{i\phi_2}|^{\alpha},$$
because if $\theta$ is the angle between $r_1e^{i\phi_1}$ and $r_2e^{i\phi_2}$, we have:
$$|r_1e^{i\phi_1}-r_2e^{i\phi_2}|^2 -| r_1e^{i\phi_1}- r_1e^{i\phi_2}|^2=r_2^2-r_1^2-2r_1r_2\cos(\theta)+2r_1^2\cos(\theta)$$
$$=(r_2-r_1)(r_1+r_2-2r_1\cos(\theta))\geq (r_2-r_1)^2\geq 0.$$
\end{proof}

\begin{proposition}   \label{prop8}
Let $g\in C_{per}^{2,\alpha}$ and $u$ as in \eqref{PKer}, then (for $1<|x|<2$):\\
$\left\|Du\right\|_{\infty} \leq C (\left\|g\right\|_{\infty}+[g]_{0,\alpha}).$\\
$  [D u]_{0,\alpha}\leq C(\left\|g\right\|_{\infty}+[g]_{0,\alpha}). $\\
$ \left\|D^2 u\right\|_{\infty} \leq C(\left\|g'\right\|_{\infty}+[g']_{0,\alpha}+\left\|g''\right\|_{\infty}+[g'']_{0,\alpha}).  $\\
$  [D^2 u]_{0,\alpha}\leq C(\left\|g'\right\|_{\infty}+[g']_{0,\alpha}+\left\|g''\right\|_{\infty}+[g'']_{0,\alpha}). $
\end {proposition}

\begin{proof}

It follows by Theorem \ref{rfthm}, Proposition \ref{prop5}, Proposition \ref{prop6}, Lemma \ref{lemma5} and Proposition \ref{prop7} (note that we have used that $[fg]_{0,\alpha}\leq \left\|f\right\|_{\infty}[g]_{0,\alpha}+\left\|g\right\|_{\infty}[f]_{0,\alpha}$).

\end{proof}
 
\begin{proposition}  \label{prop9}
Let $g\in C^{1,\alpha}(\partial B_1)$ and $u(x)=\int_{\partial B_1}g(y)\log|y-x|dS(y)$, then (for $1<|x|<2$) :\\
$\left\|Du\right\|_{\infty} \leq C (\left\|g\right\|_{\infty}+[g]_{0,\alpha}).$\\
$  [D u]_{0,\alpha}\leq C(\left\|g\right\|_{\infty}+[g]_{0,\alpha}). $\\
$ \left\|D^2 u\right\|_{\infty} \leq C(\left\|g\right\|_{\infty}+[g]_{0,\alpha}+\left\|g'\right\|_{\infty}+[g']_{0,\alpha}).  $\\
$  [D^2 u]_{0,\alpha}\leq C(\left\|g\right\|_{\infty}+[g]_{0,\alpha}+\left\|g'\right\|_{\infty}+[g']_{0,\alpha}). $
\end {proposition}

\begin{proof}:

It follows by Theorem \ref{rfthm}, Proposition \ref{prop5}, Proposition \ref{prop6}, Lemma \ref{lemma5} and Proposition \ref{prop7}.

\end{proof}

\section{H\"older regularity for the harmonic function in a holed domain}
\label{se:movingdomain}

Throughout this section we study the H\"older regularity of the classical 2D singular integrals in a 
generic annulus:
\begin{align}
 \label{eq:Omega}
 \Omega:=\{x\in \mathbb{R}^2: R<|x|<R+d\}.
\end{align}
For calculations that have to be made away from $\partial \Omega$,
we work in 
\begin{align}
 \Omega':=\{x\in \mathbb{R}^2: R+\frac{1}{3}d<|x|<R+\frac{2}{3}d\}.
\end{align}

The role of the generic length $d$ is that of giving a uniform lower bound
for the width of an annular neighbourhood
of the excised hole that is still contained in the domain.

In Proposition \ref{prop10}  negative powers 
of the radii of the holes are obtained. It is for this reason that in the final result (see \eqref{eq:d}) 
not only the distances
between the holes but also their radii are assumed to be greater than the generic length $d$.
In some intermediate results, knowing that the radius is greater than $d$  simplifies
the  estimates (e.g.\ in Lemma \ref{lemma2} we obtain 
$\|Du\|_\infty \leq CR\|f\|_\infty$ instead of $\|Du\|_\infty\leq C(R+d)\|f\|_\infty$).
This is why the hypothesis $R\geq Cd$ is added througout the whole section.

\subsection{Estimates in the interior of the domain}
\label{se:interior}

The following regularity estimates for harmonic functions can be found in \cite[Thm.\ 2.2.7]{Evans10}

\begin{lem}   \label{harm reg}
Let $v$ be harmonic in $B(x,d)$, then:\\
$ \left\|v\right\|_{L^{\infty}(B(x,\frac{d}{2}))}\leq C d^{-2}\left\|v\right\|_{L^1(B(x,d))} .$\\
$ \left\|D^{\beta}v\right\|_{L^{\infty}(B(x,\frac{d}{2}))}\leq C d^{-2-|\beta|}\left\|v\right\|_{L^1(B(x,d))} .$
\end{lem}

A careful inspection of the proof of \cite[Prop.\ 5.1]{CH19a} yields the following dependence on R and d in the Hölder interior estimates for harmonic functions.

\begin{proposition} \label{prop2}
: Let $v$ be harmonic in $\Omega$ and $R\geq Cd$, then we have the folllowing estimates :\\
 $ \left\|v\right\|_{L^{\infty}(\Omega')}\leq C d^{-2}\left\|v\right\|_{L^1(\Omega)} .$\\
$ [v  ]_{0,\alpha(\Omega')}\leq Cd^{-3}R^{1-\alpha}\left\|v\right\|_{L^1(\Omega)}.$\\
$ \left\|D^{\beta}v\right\|_{L^{\infty}(\Omega')}\leq C d^{-2-|\beta|}\left\|v\right\|_{L^1(\Omega)} .$\\
$ [ v ]_{1,\alpha(\Omega')}\leq Cd^{-4}R^{1-\alpha}\left\|v\right\|_{L^1(\Omega)}.$
\end {proposition}

\begin{lem} \label{lemma1}
Let $R\geq Cd$, $v$ be harmonic in $\Omega$ and $\zeta$ a cut-off function with support within $|x|<R+\frac{2}{3}d$ and equal to $1$ for $|x|\leq R+\frac{1}{3}d$, then:\\
$ [\Delta(v\zeta)]_{0,\alpha(\mathbb{R}^2)}\leq CR^{1-\alpha}d^{-5}\left\|v\right\|_{L^1(\Omega)}. $\\
$ \left\|\Delta(v\zeta)\right\|_{\infty (\mathbb{R}^2)}\leq Cd^{-4}  \left\|v\right\|_{L^1(\Omega)} .$
\end {lem}

\begin{proof}
 It is clear that we can choose $\zeta$ to be such that: $|D^{k}\zeta|\leq C_kd^{-k}$ (and then $[\zeta]_{k,\alpha(\Omega')}\leq C_{k+1}d^{-k-1}R^{1-\alpha} $
 since $\zeta\in C_{c}^{\infty}(B(0,R+d))$). Then, using Proposition \ref{prop1} and the estimates for $\zeta$ we get: 
$$|\Delta(v\zeta)|\leq 2|\nabla v \cdot \nabla \zeta| +|v\Delta \zeta|\leq C d^{-4}\left\|v\right\|_{L^1(\Omega)}.$$
On the other hand:
$$ [\Delta(v\zeta)]_{0,\alpha(\Omega')}\leq 2[\nabla v \cdot \nabla \zeta]_{0,\alpha(\Omega')} +[v\Delta \zeta]_{0,\alpha(\Omega')}.$$
Now note that:
$$ [v_{,\beta} \cdot \zeta_{,\beta}]_{0,\alpha(\Omega')}\leq  [v_{,\beta}]_{0,\alpha(\Omega')}\left\|\zeta_{,\beta}\right\|_{\infty(\Omega')}+[\zeta_{,\beta}]_{0,\alpha(\Omega')}\left\|v_{,\beta}\right\|_{\infty(\Omega')}$$
$$\leq Cd^{-4}R^{1-\alpha}\left\|v\right\|_{L^1(\Omega)}\cdot d^{-1}+ Cd^{-2}R^{1-\alpha}\cdot d^{-3}\left\|v\right\|_{L^1(\Omega)}.$$
Furthermore:
$$[v\Delta \zeta]_{0,\alpha(\Omega')}\leq [v]_{0,\alpha(\Omega')}\left\|\Delta\zeta\right\|_{\infty(\Omega')}+[\Delta\zeta]_{0,\alpha(\Omega')}\left\|v\right\|_{\infty(\Omega')}$$
$$ \leq Cd^{-3}R^{1-\alpha}\left\|v\right\|_{L^1(\Omega)}\cdot d^{-2}+Cd^{-3}R^{1-\alpha}\cdot d^{-2}\left\|v\right\|_{L^1(\Omega)}.$$
Hence:
$$ [\Delta(v\zeta)]_{0,\alpha(\Omega')}\leq Cd^{-5}R^{1-\alpha}\left\|v\right\|_{L^1(\Omega)}.$$
Now if $x\in \Omega'$ and $y\in \mathbb{R}^2\setminus \overline{\Omega'}$, there exists $t\in (0,1)$ such that $z=tx+(1-t)y\in \partial\Omega'$, then we have
$$ |\Delta(v\zeta)(x)-\Delta(v\zeta)(y)|\leq |\Delta(v\zeta)(x)-\Delta(v\zeta)(z)|+|\Delta(v\zeta)(z)-\Delta(v\zeta)(y)|$$
$$=|\Delta(v\zeta)(x)-\Delta(v\zeta)(z)|\leq CR^{1-\alpha}d^{-5}\left\|v\right\|_{L^1(\Omega)}|x-z|^{\alpha}$$
$$=CR^{1-\alpha}d^{-5}\left\|v\right\|_{L^1(\Omega)}(1-t)^{\alpha}|x-y|^{\alpha}\leq CR^{1-\alpha}d^{-5}\left\|v\right\|_{L^1(\Omega)}|x-y|^{\alpha}$$
(Clearly if $x,y\in \mathbb{R}^2\setminus \overline{\Omega'}$, $|\Delta(v(x)\zeta(x))-\Delta(v(y)\zeta(y))|=0$).
Finally, we get:
$$[\Delta(\zeta v)]_{0,\alpha(\mathbb{R}^2)}\leq CR^{1-\alpha}d^{-5}\left\|v\right\|_{L^1(\Omega)}. $$
\end{proof}

\subsection{Estimates near circular boundaries}
\label{se:circles}

\begin{proposition} \label{prop1}
Let $v$ be harmonic in $\Omega$ and $\zeta$ be a cut-off function with support within $|x|<R+\frac{2}{3}d$ and equal to $1$ for $|x|\leq R+\frac{1}{3}d$. 
Then, if $u=\zeta v$:
$$u(x)=C-\int_{\partial B_R}\frac{\partial u}{\partial \nu}\left( \Phi(y-x)-\phi^{x}(y) \right)dS(y)-\int_{\Omega}\Delta u \left( \Phi(y-x)-\phi^{x}(y)  \right)dy.$$ 
\end{proposition}

\begin{proof}
 Let us proceed as in \cite{Evans10}:
$$ \int_{\Omega\setminus B_{\varepsilon}(x)}\Delta u(y)\Phi(y-x)-u(y)\Delta_{y}\Phi (y-x)  dy=\int_{\partial \Omega}\frac{\partial u}{\partial \nu}\Phi(y-x)-\frac{\partial \Phi}{\partial \nu}(y-x)u(y) dS(y)$$
$$+\int_{\partial B_{\varepsilon}(x)}\frac{\partial \Phi}{\partial \nu}(y-x)u(y)-\frac{\partial u}{\partial \nu}\Phi(y-x)dS(y),$$
letting $\varepsilon \rightarrow 0$ (and using the fact that $u$ vanishes outside $B_{R+\frac{2}{3}d}$), we get:
$$ \int_{\Omega}\Delta u(y)\Phi(y-x) dy= \int_{\partial B_R}\frac{\partial \Phi}{\partial \nu}(y-x)u(y)-\frac{\partial u}{\partial \nu}\Phi(y-x) dS(y)-u(x).$$
Hence:$$ u(x)=\int_{\partial B_R}\frac{\partial \Phi}{\partial \nu}(y-x)u(y)-\frac{\partial u}{\partial \nu}\Phi(y-x) dS(y)-\int_{\Omega}\Delta u(y)\Phi(y-x) dy,$$
with the normal pointing outside $B_R$. Now (as can be seen in \cite{DiBenedetto09}),  if a function $\phi^{x}(y)$ satisfies:
\begin{align} \label{benedetto}
\left\{ \begin{aligned}
-\Delta_{y}\phi^{x}(y)&=k& &\text{if $y\in \Omega$,} \\
\frac{\partial \phi^{x}}{\partial \nu} &=\frac{\partial \Phi}{\partial \nu}(y-x) & &\text{if $y\in \partial B_R$ ,}
\end{aligned} \right. \end {align}
with $k$ being a constant, then:
$$ \int_{\Omega}\Delta_{y}\phi^{x}(y)u(y)-\Delta u \phi^{x}(y)dy=\int_{\partial \Omega}u(y)\frac{\partial}{\partial \nu}\phi^{x}(y)-\phi^{x}(y)\frac{\partial u}{\partial \nu} dS(y)$$
$$=\int_{\partial B_R}\phi^{x}(y)\frac{\partial u}{\partial \nu}-u(y)\frac{\partial}{\partial \nu}\Phi(y-x) dS(y)=k\int_{\Omega}u dy -\int_{\Omega}\Delta u \phi^{x}(y)dy,$$
\noindent where we have used \eqref{benedetto}. Finally, replacing in the expression for $u(x)$, we obtain:
$$ u(x)=C-\int_{\partial B_R}\frac{\partial u}{\partial \nu}\left( \Phi(y-x)-\phi^{x}(y) \right)dS(y)-\int_{\Omega}\Delta u \left( \Phi(y-x)-\phi^{x}(y)  \right)dy  .$$
It is easy to see that $\phi^x(y)=\frac{1}{2\pi}\log(|y-x^{*}|)-\frac{|y|^2}{4\pi R^2}$ satisfies \eqref{benedetto} 
using the identity $|x_1||x_2-x_1^{*}|=|x_2||x_1-x_2^{*}|$.
\end{proof}

\begin{proposition} \label{prop3}
Let $f\in C_c^{0,\alpha}(\Omega')$, $R\geq Cd$ and $u=\int_{\mathbb{R}^2}f(y)\Phi(x-y)dy$, then:\\
$ \left\|D u\right\|_{\infty(\mathbb{R}^2)}\leq CR \left\|f\right\|_{\infty}  .$\\
$[Du]_{0,\alpha (B(0,R+d)\setminus\overline{B(0,R)})}\leq CR^{1-\alpha}\left\|f\right\|_{\infty}$\\
$  \left\|\partial_{\beta\gamma}^{2} u\right\|_{\infty(B(0,R+d)\setminus\overline{B(0,R)})}\leq  CR^{\alpha}[f]_{0,\alpha(\mathbb{R}^2)}+\frac{\delta_{\beta\gamma}}{2}\left\|f\right\|_{\infty} .$\\
$  [D^2 u]_{0,\alpha(B(0,R+d)\setminus\overline{B(0,R)})}\leq C[f]_{0,\alpha(\mathbb{R}^2)}. $
\end {proposition}

\begin{proof} Let us estimate the first derivative:
$$|u_{,\beta}|\leq \left\|f\right\|_{\infty}\int_{\Omega'}\frac{dy}{|x-y|}\leq C\left\|f\right\|_{\infty}\int_{0}^{2R+\frac{5}{3}d}dr\leq CR\left\|f\right\|_{\infty}   ,$$
Now let us estimate the H\"older seminorm of the derivatives: let
$$v_{\rho}=\int_{\mathbb{R}^2\setminus B(x,\rho)}f(y)\Phi_{,\beta}(x-y)dy,$$
 with $\rho\in (0,2(R+d))$, then: $$|u_{,\beta}-v_{\rho}|\leq C\left\|f\right\|_{\infty}\int_{ B(x,\rho)}|x-y|^{-1}dy\leq C\left\|f\right\|_{\infty}\int_{ B(x,\rho)}|x-y|^{-1}dy$$
$$ \leq C\left\|f\right\|_{\infty}\rho\leq C\left\|f\right\|_{\infty}\rho^{\alpha}R^{1-\alpha}.$$
On the other hand:
$$\frac{\partial v_{\rho}}{\partial \gamma}=\int_{\mathbb{R}^2\setminus B(x,\rho)}f(y)\Phi_{,\beta\gamma}(x-y)dy-\int_{\partial B(x,\rho)}f(y)\Phi_{,\beta}(x-y)\nu_{\gamma}dS(y) ,$$
therefore: 
$$ \left|\frac{\partial v_{\rho}}{\partial \gamma}\right|\leq C\left\|f\right\|_{\infty}\left(  \int_{\mathbb{R}^2\setminus B(x,\rho)}|x-y|^{-2}dy+\int_{\partial B(x,\rho)}|x-y|^{-1}dS(y)\right)$$
$$\leq  C\left\|f\right\|_{\infty}\left(1+  \int_{B(x,2(R+d))\setminus B(x,\rho)}|x-y|^{-2}dy \right) $$
$$\leq C\left\|f\right\|_{\infty}\left(1+  \left|\log\left(\frac{R}{\rho}\right)\right| \right)\leq  C\left\|f\right\|_{\infty}\left(1+ \left(\frac{R}{\rho}\right)^{1-\alpha} \right).$$
(Note that $\frac{R}{\rho}\in(\frac{1}{2},\infty)$). Finally, if $|x-y|=\rho$:
$$|u_{,\beta}(x)-u_{,\beta}(y)|\leq |u_{,\beta}(x)-v_{\rho}(x)|+|v_{\rho}(x)-v_{\rho}(y)|+|v_{\rho}(y)-u_{,\beta}(y)|$$
$$\leq C\left\|f\right\|_{\infty}\rho^{\alpha}R^{1-\alpha}+C|x-y|\left\|f\right\|_{\infty}\left(1+ \left(\frac{R}{\rho}\right)^{1-\alpha}\right)$$
$$\leq C\left\|f\right\|_{\infty}\rho^{\alpha}R^{1-\alpha},$$
where we have used that $\rho\leq CR$.

To prove the third estimate, first note that the second derivatives of $u$ are given by:
$$u_{,\beta\gamma}=\lim_{\rho\rightarrow 0^+}\int_{\mathbb{R}^2\setminus B(x,\rho)}\Phi_{,\beta\gamma}(x-y)f(y)dy-\frac{\delta_{\beta\gamma}}{2}f.$$
 Since $f\in C_{c}^{0,\alpha}$ (and using the fact that $\int_{\partial B(0,1)}\Phi_{,\beta\gamma}(z)dS(z)=0$, and $\int_{A}\Phi_{,\beta\gamma}(z)dz= 0$
 if $A$ is any annulus centered at the origin), the absolute value of the singular integral is bounded by:
$$  \left| \lim_{\rho\rightarrow 0^+}\int_{B(x,2R+\frac{5}{3}d)\setminus B(x,\rho)}(f(y)-f(x))\Phi_{,\beta\gamma}(x-y)dy \right|$$
$$\leq \lim_{\rho\rightarrow 0^+}
\int_{\partial B(0,1)}|\Phi_{,\beta\gamma}(\omega)|dS(\omega)\int_{\rho}^{2R+\frac{5}{3}d}r^{\alpha-1}dr[f]_{0,\alpha}\leq CR^{\alpha}[f]_{0,\alpha};$$
that proves the second result (obviously we have $\left\|\frac{\delta_{ij}}{2}f\right\|_{\infty}\leq \frac{\delta_{ij}}{2}\left\|f\right\|_{\infty}$). 
To prove the last estimate, we proceed as in \cite[Thm.\ 2.6.4]{Morrey66}:
first note that if $\Phi_{,ij}(x)=\Delta(x)$, $\omega(x)=u_{,ij}(x)+\frac{\delta_{ij}}{n}f(x)$, $n=2$, and 
$$\omega_{\rho}(x)=\int_{\mathbb{R}^n\setminus B(x,\rho)}\Delta(x-\xi)f(\xi)d\xi,$$
then:
$$|\omega_{\sigma}(x)-\omega_{\rho}(x)|\leq \int_{B(x,\rho)\setminus B(x,\sigma)}|\Delta(x-\xi)|[f]_{0,\alpha}|x-\xi|^{\alpha}d\xi\leq CM_0[f]_{0,\alpha}\rho^{\alpha},$$
being $M_0=\sup_{|x|=1}|\Delta(x)|$. If we let $\sigma \rightarrow 0$, we obtain:
$$|\omega(x)-\omega_{\rho}(x)|\leq CM_0[f]_{0,\alpha}\rho^{\alpha}.$$
Let $M=3R+3d$ and  $M_1=\sup_{|x|=1}|\nabla\Delta(x)|$. The derivatives of $\omega_{\rho}$ are given by:
$$\omega_{\rho,\beta}(x)=\int_{\mathbb{R}^n\setminus B(x,\rho)}\Delta_{,\beta}(x-\xi)f(\xi)d\xi-\int_{\partial B(x,\rho)}\Delta(x-\xi)f(\xi)d\xi_{\beta}^{'}$$
$$= \int_{B(x,M)\setminus B(x,\rho)}\Delta_{,\beta}(x-\xi)(f(\xi)-f(x))d\xi+\int_{\partial B(x,M)}\Delta(x-\xi)(f(\xi)-f(x))d\xi_{\beta}^{'}$$
$$+\int_{\partial B(x,\rho)}\Delta(x-\xi)(f(x)-f(\xi))d\xi_{\beta}^{'}$$
Note that: $$\int_{\partial B(x,M)}\Delta(x-\xi)f(\xi)d\xi_{\beta}^{'}=0.$$
Let $x,z\in B(0,R+d)$ and $\rho=|x-z|$,then:
$$|\nabla \omega_{\rho}|\leq C(M_0+M_1)[f]_{0,\alpha}(\rho^{\alpha-1}+M^{\alpha-1})\leq C(M_0+M_1)[f]_{0,\alpha}\rho^{\alpha-1}.$$
Thus (applying the mean value theorem):
$$|\omega(x)-\omega(z)|\leq |\omega(x)-\omega_{\rho}(x)|+|\omega_{\rho}(x)-\omega_{\rho}(z)|
  +|\omega_{\rho}(z)-\omega(z)|\leq C(M_0+M_1)[f]_{0,\alpha}\rho^{\alpha};$$
that yields: $[\omega]_{0,\alpha}\leq C(M_0+M_1)[f]_{0,\alpha}$.
\end{proof}

\begin{lem} \label{lemma2}
Let $u=\int_{\mathbb{R}^2}f(y)\log|x^{*}-y|dy$ with $f\in C_c^{0,\alpha}(B_{R+\frac{2}{3}d}\setminus\overline{B_{R+\frac{d}{3}}})$, $R\geq Cd$.  
Then:\\
$\left\| Du \right\|_{L^{\infty}(B_{R+d}\setminus \overline{B_R})}\leq C R\left\|f\right\|_{\infty}$.\\
$[ D u]_{0,\alpha(B_{R+d}\setminus \overline{B_R})} \leq CR^{2-\alpha}d^{-1}\left\|f\right\|_{\infty}$.\\
$\left\| D^2 u \right\|_{L^{\infty}(B_{R+d}\setminus \overline{B_R})} \leq CRd^{-1}\left\|f\right\|_{\infty}$.\\
$[ D^2 u]_{0,\alpha(B_{R+d}\setminus \overline{B_R})} \leq CR^{2-\alpha}d^{-2}\left\|f\right\|_{\infty}$.
\end {lem}

\begin{proof}
 Using the identity $|x_1||x_1^{*}-x_2|=|x_2||x_1-x_2^{*}|$, let us first note that:
\begin{align} \label{log-reflection}
\log|y-x^{*}|=\log|y^{*}-x|+\log|y|-\log|x|,  \end{align}
this implies that:
$$  u=C+\int_{\mathbb{R}^2}\log|x-y^{*}|f(y)dy-\log|x|\int_{\mathbb{R}^2}f(y)dy ,$$
\noindent then:
$$  |u_{,\beta}|\leq C\int_{\Omega'}\frac{|f(y)|dy}{|x-y^{*}|} + \frac{C}{|x|}\left\|f\right\|_{\infty}Rd\leq C\int_{\Omega'}\frac{|f(y)|dy}{|x|-|y^{*}|} + \frac{C}{|x|}\left\|f\right\|_{\infty}Rd$$
$$\leq CRd\frac{\left\|f\right\|_{\infty}}{R-\frac{R^2}{R+\frac{d}{3}}}+Cd\left\|f\right\|_{\infty}\leq CR\left\|f\right\|_{\infty}.$$
\noindent The other estimates are proved analogously (for the H\"older continuity we can use the same argument as in \cite[Prop.\ 5.1]{CH19a}).
\end{proof}

\begin{proposition} \label{prop4}
Let $f\in C_c^{0,\alpha}(B_{R+\frac{2}{3}d}\setminus\overline{B_{R+\frac{d}{3}}})$, $R\geq Cd$ and $u=\int_{\mathbb{R}^2}f(y)G_N(x,y)dy$, then (in $B_{R+d}\setminus \overline{B_R}$) :\\
$ \left\|Du\right\|_{\infty}\leq CR\left\|f\right\|_{\infty} .$\\
$ [D u]_{0,\alpha} \leq CR^{2-\alpha}d^{-1}\left\|f\right\|_{\infty}.$\\
$ \left\|D^2u\right\|_{\infty} \leq C(Rd^{-1}\left\|f\right\|_{\infty}+R^{\alpha}[f]_{0,\alpha}).$\\
$ [D^2 u]_{0,\alpha} \leq C(R^{2-\alpha}d^{-2}\left\|f\right\|_{\infty}+[f]_{0,\alpha}).$
\end {proposition}

\begin{proof} It follows from Proposition \ref{prop3} and  Lemma \ref{lemma2}. 
\end{proof}

\begin{proposition}  \label{prop10}
Let $g\in C^{1,\alpha}(\partial B_R)$ and $u=\int_{\partial B_R}g\log|y-x|dS$, then (for $R<|x|<R+d$, with $d\leq R$) :\\
$\left\|Du\right\|_{\infty} \leq C (\left\|g\right\|_{\infty}+R^{\alpha}[g]_{0,\alpha}).$\\
$  [D u]_{0,\alpha}\leq C(R^{-\alpha}\left\|g\right\|_{\infty}+[g]_{0,\alpha}). $\\
$ \left\|D^2 u\right\|_{\infty} \leq C(R^{-1}\left\|g\right\|_{\infty}+R^{\alpha-1}[g]_{0,\alpha}+\left\|g'\right\|_{\infty}+R^{\alpha}[g']_{0,\alpha}).  $\\
$  [D^2 u]_{0,\alpha}\leq C(R^{-1-\alpha}\left\|g\right\|_{\infty}+R^{-1}[g]_{0,\alpha}+R^{-\alpha}\left\|g'\right\|_{\infty}+[g']_{0,\alpha}). $\\
\end {proposition}

\begin{proof}
 It follows by a rescaling argument.
\end{proof}

\begin{proposition}   \label{prop11}
Let $u=\int_{\partial B_R}gG_N(x,y)dS(y)$, then:\\
$\left\|Du\right\|_{\infty(B(0,R+d)\setminus\overline{B(0,R)})} \leq C (\left\|g\right\|_{\infty}+R^{\alpha}[g]_{0,\alpha}).$\\
$  [D u]_{0,\alpha(B(0,R+d)\setminus\overline{B(0,R)})}\leq C(R^{-\alpha}\left\|g\right\|_{\infty}+[g]_{0,\alpha}) .$\\
$ \left\|D^2 u\right\|_{\infty(B(0,R+d)\setminus\overline{B(0,R)})} \leq C(R^{-1}\left\|g\right\|_{\infty}+R^{\alpha-1}[g]_{0,\alpha}+\left\|g'\right\|_{\infty}+R^{\alpha}[g']_{0,\alpha}) . $\\
$  [D^2 u]_{0,\alpha(B(0,R+d)\setminus\overline{B(0,R)})}\leq C(R^{-1-\alpha}\left\|g\right\|_{\infty}+R^{-1}[g]_{0,\alpha}+R^{-\alpha}\left\|g'\right\|_{\infty}+[g']_{0,\alpha}) .$
\end {proposition}

\begin{proof}
 Thanks to \eqref{log-reflection} we have: 
$$G_N(x,y)=-\frac{1}{\pi}\log|y-x|+\frac{1}{2\pi}\log\frac{|x|}{R}-\frac{|y|^2}{4\pi R^2}. $$
The estimates for $u$ then follow from Proposition \ref{prop10} and estimates for $\log|x|$ 
(for the H\"older continuity, we can proceed as in \cite[Prop.\ 5.1]{CH19a}).
\end{proof}

  \subsection{A trace theorem and the $L^1$ norm}

The proofs of the following two results can be found on \cite[Lemma 5.2, Prop.\ 5.5]{CH19a}.

\begin{lem} \label{trace}
Let $\phi\in H^{1}(B_{\rho_2}\setminus \overline{B_{\rho_1}})$ for some $0<\rho_1<\rho_2$. Then (for $i=1,2$):
$$\int_{\partial B_{\rho_i}}\phi^2(x)dS(x)\leq   
\frac{8}{\rho_2-\rho_1}\int_{B_{\rho_2}\setminus \overline{B_{\rho_1}}}\phi^2(x)dx+4(\rho_2-\rho_1)\int_{B_{\rho_2}\setminus \overline{B_{\rho_1}}}|D\phi|^2(x)dx.  $$
\end{lem}

\begin{proposition} \label{prop12}
Let $E$, $d$, and $B$ be as in \eqref{eq:genericE}, \eqref{eq:d},
and \eqref{eq:B}.
Suppose
\[ 
\begin{cases}
\Delta u =0 \text{ in } E,\\[0.5em]
\displaystyle \frac{\partial u}{\partial \nu}=g
 \text{ on } 
\partial E
\end{cases}
\]
and $\displaystyle \int_{E}u(y)dy=0$.
Then: $ \Vert u \Vert_{L^{1}(E)}\leq C\cdot B\Vert g\Vert_{\infty}.   $
\end {proposition}

 \subsubsection*{Regularity near the holes}
\begin{proposition}  \label{prop13}
Let $B$ and $u$ be as in \eqref{eq:B} and Proposition \ref{prop12}, then, if $A=\cup_{k=1}^{n}B(z_k,r_k+\frac{d}{3})\setminus\overline{B(z_k,r_k)}$, we have:\\
$\Vert Du \Vert_{L^{\infty}(A)}\leq C\left(1+Bd^{-4}r_0\right)\Vert g \Vert_{\infty}+Cr_{0}^{\alpha}[g]_{0,\alpha}.$ \\
$[Du]_{0,\alpha(B(z_k,r_k+\frac{d}{3})\setminus\overline{B(z_k,r_k)})}\leq C\left(Bd^{-5}r_0^{2-\alpha}+d^{-\alpha}\right)\Vert g \Vert_{\infty}+C[g]_{0,\alpha}.$  \\
$\Vert D^2u \Vert_{L^{\infty}(A)}\leq C\left(Bd^{-5}r_0+d^{-1}\right)\Vert g \Vert_{\infty}+Cd^{\alpha-1}[g]_{0,\alpha}+C\Vert g' \Vert_{\infty}+Cr_{0}^{\alpha}[g']_{0,\alpha}.$\\
$[D^2u]_{0,\alpha(B(z_k,r_k+\frac{d}{3})\setminus\overline{B(z_k,r_k)})}\leq C\left(Bd^{-6}r_{0}^{2-\alpha}+d^{-1-\alpha}\right)\Vert g \Vert_{\infty}+Cd^{-1}[g]_{0,\alpha}+Cd^{-\alpha}\Vert g' \Vert_{\infty}+C[g']_{0,\alpha}.$\\
\end {proposition}

\begin{proof} 
 It follows from Proposition \ref{prop1}, Proposition \ref{prop11}, 
Proposition \ref{prop4}, Lemma \ref{lemma1} and Proposition \ref{prop12} (recall that $r_i\geq d$).
\end{proof}
 
 \subsubsection*{Interior regularity}
\begin{proposition}
  \label{prop14}
Let $E$, $d$, and $B$ be as in \eqref{eq:genericE}, \eqref{eq:d},
and \eqref{eq:B}. Let $u$ be harmonic in $E$ and $E'=B(z_0,r_0-\frac{d}{3})\setminus \bigcup_{k=1}^{n}B(z_k,r_k+\frac{d}{3})$, then:\\
\noindent$ \left\|u\right\|_{L^{\infty}(E')}\leq C d^{-2}\left\|u\right\|_{L^1(E)}\leq  C Bd^{-2} \left\|g\right\|_{\infty}.$\\
$ [u  ]_{0,\alpha(E')}\leq Cd^{-3}r_{0}^{1-\alpha}\left\|u\right\|_{L^1(E)}\leq CBd^{-3}r_{0}^{1-\alpha}\left\|g\right\|_{\infty}.$\\
$ \left\|D^{\beta}u\right\|_{L^{\infty}(E')}\leq C d^{-2-|\beta|}\left\|u\right\|_{L^1(E)}\leq  C Bd^{-2-|\beta|}\left\|g\right\|_{\infty} .$\\
$ [ u ]_{1,\alpha(E')}\leq Cd^{-4}r_{0}^{1-\alpha}\left\|u\right\|_{L^1(E)}\leq CBd^{-4}r_{0}^{1-\alpha}\left\|g\right\|_{\infty}.$\\
$ [ D^{2}u ]_{0,\alpha(E')}\leq Cd^{-5}r_{0}^{1-\alpha}\left\|u\right\|_{L^1(E)}\leq CBd^{-5}r_{0}^{1-\alpha}\left\|g\right\|_{\infty}.$
\end {proposition}

\begin{proof} 
 It follows from local regularity for harmonic functions and Proposition \ref{prop2} (using triangle inequality at most $2n+1$ times): 
 join $x$ and $z$ with a straight line, 
 then the segment intersects at most the $n$ holes. 
 In that case, join the points 
 using segments of the above straight line and segments of circles of the form $\partial B(z_k,r_k+\frac{d}{3})$ 
 (for straight lines use local estimates for harmonic functions and for circles use Proposition \ref{prop2}).
\end{proof}
 
 \subsubsection*{Regularity near the exterior boundary}
 
 In the next proposition and lemma, $R$ should be thought of as $r_0-d$, hence $\{x: R<|x|<R+d\}$ is the the part of the $d$-neighbourhood of the exterior boundary
 that lies inside $E$.
 
\begin{proposition} \label{prop15} 
Let $v$ be harmonic in $\Omega$ 
and $\zeta$ be a cut-off function equal to $0$ for $|x|\leq R+\frac{d}{3}$ and equal to $1$ for $R+\frac{2}{3}d\leq |x|$, then, if $u=\zeta v$:
$$u(x)=C+\int_{\partial B_R}\frac{\partial u}{\partial \nu}\left( \Phi(y-x)-\phi^{x}(y) \right)dS(y)-\int_{\Omega}\Delta u \left( \Phi(y-x)-\phi^{x}(y)  \right)dy.$$ 
\end {proposition}

\begin{proof} 
 This can be shown using the same techniques as in the proof of Proposition \ref{prop1}. 
\end{proof}
 
 The proofs of the following two results, are similar to the proof of Lemma \ref{lemma1} and Proposition \ref{prop11}, respectively:
\begin{lem} \label{lemma16}
Let $R\geq Cd$, $v$ be harmonic in $\Omega$ and $\zeta$ be a cut-off function equal to $0$ for $|x|\leq R+\frac{d}{3}$
and equal to $1$ for $R+\frac{2}{3}d\leq |x|$, then:\\
$ [\Delta(v\zeta)]_{0,\alpha(\mathbb{R}^2)}\leq CR^{1-\alpha}d^{-5}\left\|v\right\|_{L^1(\Omega)}. $\\
$ \left\|\Delta(v\zeta)\right\|_{\infty (\mathbb{R}^2)}\leq Cd^{-4}  \left\|v\right\|_{L^1(\Omega)}.$
\end {lem}

\begin{proposition}   \label{prop16}
Let $u=\int_{\partial B_{r_0}}gG_N(x,y)dS(y)$, then:\\
$\left\|Du\right\|_{\infty(B(0,r_0)\setminus\overline{B(0,r_0-\frac{d}{3})})} \leq C (\left\|g\right\|_{\infty}+r_{0}^{\alpha}[g]_{0,\alpha}).$\\
$  [D u]_{0,\alpha(B(0,r_0)\setminus\overline{B(0,r_0-\frac{d}{3})})}\leq C(r_{0}^{-\alpha}\left\|g\right\|_{\infty}+[g]_{0,\alpha}) .$\\
$ \left\|D^2 u\right\|_{\infty(B(0,r_0)\setminus\overline{B(0,r_{0}-\frac{d}{3})})} \leq C(r_{0}^{-1}\left\|g\right\|_{\infty}+r_{0}^{\alpha-1}[g]_{0,\alpha}+\left\|g'\right\|_{\infty}+r_{0}^{\alpha}[g']_{0,\alpha}) . $\\
$  [D^2 u]_{0,\alpha(B(0,r_0)\setminus\overline{B(0,r_{0}-\frac{d}{3})})}\leq C(r_{0}^{-1-\alpha}\left\|g\right\|_{\infty}+r_{0}^{-1}[g]_{0,\alpha}+r_{0}^{-\alpha}\left\|g'\right\|_{\infty}+[g']_{0,\alpha}) .$
\end {proposition}

\begin{proposition} \label{prop17}
Let $B$ and $u$ be as in Proposition \ref{prop12}, then, we have:\\
$\left\|Du\right\|_{\infty(B(0,r_0)\setminus\overline{B(0,r_0-\frac{d}{3})})} \leq C (1+Bd^{-4}r_0)\left\|g\right\|_{\infty}+Cr_{0}^{\alpha}[g]_{0,\alpha}.$\\
$  [D u]_{0,\alpha(B(0,r_0)\setminus\overline{B(0,r_0-\frac{d}{3})})}\leq C(r_{0}^{-\alpha}+Bd^{-5}r_{0}^{2-\alpha})\left\|g\right\|_{\infty}+C[g]_{0,\alpha} .$\\
$ \left\|D^2 u\right\|_{\infty(B(0,r_0)\setminus\overline{B(0,r_{0}-\frac{d}{3})})} \leq C(r_{0}^{-1}+Bd^{-5}r_0)\left\|g\right\|_{\infty}+Cr_{0}^{\alpha-1}[g]_{0,\alpha}+C\left\|g'\right\|_{\infty}+Cr_{0}^{\alpha}[g']_{0,\alpha} . $\\
$  [D^2 u]_{0,\alpha(B(0,r_0)\setminus\overline{B(0,r_{0}-\frac{d}{3})})}
  \leq C(r_{0}^{-1-\alpha}+Bd^{-6}r_{0}^{2-\alpha})\left\|g\right\|_{\infty}+Cr_{0}^{-1}[g]_{0,\alpha}+Cr_{0}^{-\alpha}\left\|g'\right\|_{\infty}+C[g']_{0,\alpha} .$
\end {proposition}

\begin{proof} 
First note that the hypothesis: $B(z_i, r_i+d)\subset B(z_0, r_0)$ and $r_i\geq d$ for all $i\in \{1,\ldots, n\}$, implies that $r_0\geq 2d$. Hence, the hypothesis $R\geq Cd$ for some $C>0$
is satisfied when $R=r_0-d$. The estimates then follow from Proposition \ref{prop15}, Proposition \ref{prop16}, Proposition \ref{prop4}, 
Lemma \ref{lemma16} and Proposition \ref{prop12}. 
\end{proof}
 
\subsubsection*{Global regularity}
\begin{thm} \label{thm1}
Let $B$ and $u$ be as in Proposition \ref{prop12}, then, we have:\\
$\left\|Du\right\|_{\infty(E)} \leq C (1+Bd^{-4}r_0)\left\|g\right\|_{\infty}+Cr_{0}^{\alpha}[g]_{0,\alpha}.$\\
$  [D u]_{0,\alpha(E)}\leq C(d^{-\alpha}+Bd^{-5}r_{0}^{2-\alpha})\left\|g\right\|_{\infty}+C[g]_{0,\alpha} .$\\
$ \left\|D^2 u\right\|_{\infty(E)} \leq C(d^{-1}+Bd^{-5}r_0)\left\|g\right\|_{\infty}+Cd^{\alpha-1}[g]_{0,\alpha}+C\left\|g'\right\|_{\infty}+Cr_{0}^{\alpha}[g']_{0,\alpha} . $\\
$  [D^2 u]_{0,\alpha(E)}
\leq C(d^{-1-\alpha}+Bd^{-6}r_{0}^{2-\alpha})\left\|g\right\|_{\infty}+Cd^{-1}[g]_{0,\alpha}+Cd^{-\alpha}\left\|g'\right\|_{\infty}+C[g']_{0,\alpha}.$
\end {thm}
\begin{proof} 
 It follows from Proposition \ref{prop13}, Proposition \ref{prop14} and Proposition \ref{prop17}.
\end{proof}

\subsection*{Acknowledgments}

We are indebted to Sergio Conti, Mat\'ias Courdurier, Manuel del Pino,
Robert Kohn, Giuseppe Mingione, 
Tai Nguyen and Sylvia Serfaty for our discussions 
and their suggestions. 
This research was supported by 
the FONDECYT projects 1150038 and 1190018 of the Chilean Ministry of Education 
and by the
Millennium Nucleus Center for Analysis of PDE NC130017 
of the Chilean Ministry of Economy.

\bibliography{biblio} \bibliographystyle{alpha}

\end{document}